\documentclass[letterpaper,12pt]{amsart}

\usepackage[all]{xy}                        %

\usepackage[margin=1in]{geometry}

\CompileMatrices                            % Faster

\UseTips                                    % Use

\input xypic
\usepackage[bookmarks=true]{hyperref}       % Hyperref
\usepackage{amssymb,latexsym,amsmath,amscd}
\usepackage{xspace}
\usepackage{color}
\usepackage{graphicx}
\usepackage{caption}

\usepackage{mathtools}
\usepackage{tikz}
\usetikzlibrary{cd}

\usepackage[utf8]{inputenc}
\usepackage{fourier}
\usepackage{array}
\usepackage{makecell}

\usepackage{array}
\newcolumntype{L}[1]{>{\raggedright\let\newline\\\arraybackslash\hspace{0pt}}m{#1}}
\newcolumntype{C}[1]{>{\centering\let\newline\\\arraybackslash\hspace{0pt}}m{#1}}
\newcolumntype{R}[1]{>{\raggedleft\let\newline\\\arraybackslash\hspace{0pt}}m{#1}}

\DeclareUnicodeCharacter{02B9}{\ensuremath{'}}
\reversemarginpar

\theoremstyle{plain}
\newtheorem{theorem}{Theorem}[section]
\newtheorem*{theorem*}{Theorem}
\newtheorem{proposition}[theorem]{Proposition}
\newtheorem{corollary}[theorem]{Corollary}
\newtheorem{lemma}[theorem]{Lemma}

\theoremstyle{definition}

\newtheorem{remark}[theorem]{Remark}
\newtheorem{example}[theorem]{Example}

\newcommand{\enm}[1]{\ensuremath{#1}}          %
\newcommand{\op}[1]{\operatorname{#1}}
\newcommand{\cal}[1]{\mathcal{#1}}

\newcommand{\CC}{\enm{\mathbb{C}}}

\newcommand{\QQ}{\enm{\mathbb{Q}}}
\newcommand{\ZZ}{\enm{\mathbb{Z}}}
\newcommand{\FF}{\enm{\mathbb{F}}}

\newcommand{\PP}{\enm{\mathbb{P}}}

\newcommand{\Cc}{\enm{\cal{C}}}
\newcommand{\Dd}{\enm{\cal{D}}}
\newcommand{\Ee}{\enm{\cal{E}}}
\newcommand{\Ff}{\enm{\cal{F}}}
\newcommand{\Gg}{\enm{\cal{G}}}
\newcommand{\Hh}{\enm{\cal{H}}}
\newcommand{\Ii}{\enm{\cal{I}}}

\newcommand{\Ll}{\enm{\cal{L}}}

\newcommand{\Nn}{\enm{\cal{N}}}
\newcommand{\Oo}{\enm{\cal{O}}}

\newcommand{\Qq}{\enm{\cal{Q}}}

\newcommand{\Tt}{\enm{\cal{T}}}

\renewcommand{\phi}{\varphi}
\renewcommand{\theta}{\vartheta}
\renewcommand{\epsilon}{\varepsilon}

\newcommand{\Pic}{\op{Pic}}

\newcommand{\Ext}{\op{Ext}}

\newcommand{\End}{\op{End}}

\newcommand{\tensor}{\otimes}         % Symbols with
\newcommand{\dirsum}{\oplus}

\newcommand{\Dirsum}{\bigoplus}

\newcommand{\intsec}{\cap}

\renewcommand{\to}[1][]{\xrightarrow{\ #1\ }}

\newcommand{\old}[1]{}

\begin{document}

\title[Logarithmic vector bundles on the blown-up surface]{Logarithmic vector bundles on the blown-up surface}

\author{Sukmoon Huh} 
\address{Sungkyunkwan University, Suwon 440-746, Korea}
\email{sukmoonh@skku.edu}

\author{Min Gyo Jeong}
\address{Sungkyunkwan University, Suwon 440-746, Korea}
\email{nida08@skku.edu}

\thanks{SH is supported by the National Research Foundation of Korea(NRF) grant funded by the Korea government(MSIT) (No. 2018R1C1A6004285 and No. 2016R1A5A1008055).}



\maketitle

\begin{abstract}
We study the logarithmic vector bundles associated to arrangements of smooth irreducible curves with small degree on the blow-up of the projective plane at one point. We then investigate whether they are Torelli arrangements, that is, they can be recovered from the attached logarithmic vector bundles. 
\end{abstract}

\section{introduction}
In this paper, we study logarithmic vector bundles on the blow-up of a smooth projective surface, mostly in the case when the surface is the projective plane. For a given arrangement $\Dd$ of reduced and effective divisors on a smooth projective variety $X$, one can consider a sheaf $\Omega_X^1(\log \Dd)$ of differential $1$-forms with logarithmic poles along $\Dd$, called the {\it logarithmic sheaf} associated to $\Dd$. It is first introduced by P.~Deligne in \cite{deligne1971theorie} to define the mixed Hodge structure on the complement of $\Dd$. When the associated divisors are arranged with simple normal crossings, the sheaf turns out to be locally free and we call it {\it logarithmic vector bundle}. In this article, we assume that the divisors in consideration are arranged with simple normal crossings, unless otherwise specified. 

In \cite{dolgachev1993arrangements}, I.~Dolgachev and M.~Kapranov studied the logarithmic vector bundles associated to arrangements of hyperplanes $\Hh=\{H_1,\ldots,H_m\}$ in the projective space $\PP^{n}$. They are interested in the question when $\Hh$ can be recovered from $\Omega^1_{\PP^n}(\log\Hh)$; this kind of problem is called the \textit{Torelli problem}. The same problem can be considered for more general setting with arrangements $\Dd$ of divisors, and if $\Dd$ can be recovered from $\Omega_X^1(\log \Dd)$, then we call it a \textit{Torelli arrangement}. 

For example, in \cite{dolgachev1993arrangements} it was proven that a general hyperplane arrangement $\Hh$ in $\PP^n$ is a Torelli arrangement if $m\ge 2n+1$. More specifically, $\Hh$ is a Torelli arrangement if and only if the hyperplanes in $\Hh$ does not osculate any rational normal curve in $\PP^n$. Later the result is extended to the case $m\geq n+3$ by J.~Vall\`es in \cite{valles2000nombre}. Moreover, the associated logarithmic vector bundle $\Omega_{\PP^n}^1(\log \Hh)$ is stable with Chern polynomial $(1-t)^{n-m+1}$. Thus the Torelli arrangements on $\PP^n$ define a generically injective map from the family of arrangements to a moduli space of stable vector bundles, e.g. if $n=2$, then the induced map 
\begin{align*}
\Phi : ~~~\mathbf{F} &~~\longrightarrow ~~\mathbf{M}(c_1, c_2)\\
\Dd &\hspace{.1cm}\mapsto ~~\Omega_{\PP^2}^1(\log \Dd)
\end{align*}
is generically injective, where $\mathbf{F}$ is the family of arrangements of $m\geq 6$ lines in $\PP^2$ and $\mathbf{M}(c_1, c_2)$ is the moduli space of stable vector bundles with appropriate Chern classes. 

After that, many studies follow with more generalized and different settings. In \cite{kazushi2009logarithmic}, K.~Ueda and M.~Yosingwa suggestesd an equivalent condition for a single smooth hypersurface of degree $d\geq 3$ in $\PP^n$ to be a Torelli arrangement, that is, the defining polynomial of the hypersurface is not of Sebastiani-Thom type. In \cite{angelini2014logarithmic}, E.~Angelini considered arrangements of $\ell$ hypersurfaces $\{D_1,\ldots, D_{\ell}\}$ of the same degree $d\geq 2$ in $\PP^n$. She showed that if $\ell\geq N+4=\binom{n+d}{d} +3$, then it is a Torelli arrangement unless the hyperplanes in $\PP^N$ corresponding to $D_1,\ldots, D_{\ell}$ through the {\it Veronese map} of degree $d$ do not osculate a rational normal curve of degree $N$ in $\PP^N$.
In \cite{ballico2015torelli} E.~Ballico, S.~Huh and F.~Malaspina studied the arrangements of hyperplane sections in an $n$-dimensional smooth quadric hypersurface $\QQ_n$, to show that an arrangement of sufficiently many hyperplane sections is a Torelli arrangement. Recently in \cite{HMPV}
a generalized version of logarithmic vector bundle is introduced and its Torelli problem is investigated. 

In this article, as another attempt with the same direction, we focus on the logarithmic vector bundles on the blow-up of a smooth projective surface. In Section $2$ we provide our main theorem asserting the relation between the logarithmic vector bundles on a smooth projective surface $X$ and its blow-up $\widetilde{X}$; see Theorem \ref{main}. Using this main theorem, we obtain a positive answer to the Torelli problem for logarithmic vector bundles on the Del Pezzo surface $\FF=\PP(\Oo_{\PP^1}\oplus \Oo_{\PP^1}(-1))$ of degree $8$ associated to arrangements of rational normal curves of degree $3$. Then in Section $3$ we consider arrangements of lower degree curves on $\FF$, intersecting the exceptional divisor $E$.

%%%%%%%%%%%%%%%%%%%%%%%%%%%%%%%%%%%

\section{Logarithmic vector bundle on blown-up surfaces}
Let $X$ be a smooth projective variety of dimension $n\ge 2$ over the field of complex numbers. Let us denote by $\Dd=\{D_1, \dots, D_m\}$ be an arragement of irreducible divisors $D_i$'s on $X$. Then one can define the sheaf of differential $1$-forms with logarithmic poles along $\Dd$, denoted by $\Omega_X^1(\log \Dd)$ and called the {\it logarithmic sheaf}; refer to \cite{deligne1971theorie}. In general, the sheaf $\Omega_X^1(\log \Dd)$ is reflexive. When $\Dd$ has simple normal crossings, then the logarithmic sheaf turns out to be locally free and it fits into an exact sequence
\[
0\to \Omega^1_{X} \to \Omega^1_{X}(\log \Dd) \to \Dirsum_{i=1}^{m}\Oo_{D_i}\to 0,
\]
called the {\it Poincar\'e residue sequence}. In this case, if $z_1\cdots z_{\ell}=0$ is the local defining equation of $\Dd$ near $x\in X$, then the section of $\Omega_X^1(\log \Dd)$ are meromorphic differential forms 
\[
\omega + \sum_{i=1}^l
u_i d\log z_i
\]
with a $1$-form $\omega$ and functions $u_i$, where $\langle z_1, \dots, z_n\rangle$ is the local coordinates.

\begin{remark}
For a sub-arrangement $\Dd'=\{D_1, \dots, D_k\}\subset \Dd$ with $k<m$, one can also have a generalized version of the Poincar\'e residue sequence
\[
0\to \Omega_X^1(\log \Dd') \to \Omega_X^1(\log \Dd ) \to \bigoplus_{i=k+1}^m \Oo_{D_i} \to 0.
\]
\end{remark}

From now on we assume $n=2$, i.e., $X$ is a smooth projective surface. For a fixed finite set of distinct points $Z=\{p_1, \dots, p_k\}\subset X$, denote by $\eta:\widetilde{X}\rightarrow X$ the blow-up of $X$ along $Z$ with the exceptional divisor $E_j:=\eta^{-1}(p_j)$ for each $j$. Setting $E=E_1+\dots + E_k$, we have an exact sequence
\begin{equation}\label{OmegaPullback}
    0\to\eta^{*}\Omega^1_{X} \to \Omega^1_{\widetilde{X}} \to \Omega^1_{E}\cong\Oo_{E}(-2) \to 0;
\end{equation}
refer to \cite[Prop. II.8.11]{hartshorne2013algebraic}. Tensoring the sequence \eqref{OmegaPullback} with $\Oo_{E}$, we get
\[
0\to Tor^1_{\Oo_{\widetilde{X}}}(\Oo_{E}(-2),\Oo_{E})\cong\Oo_{E}(-1)\to \left( \eta^{*}\Omega^1_{X}\right)_{|E}\to \left (\Omega^1_{\widetilde{X}}\right)_{|E}\to\Oo_{E}(-2)\to 0.
\]
Since $\left( \eta^{*}\Omega^1_{X}\right)_{|E}\cong\Oo_{E}\oplus\Oo_{E}$ on $E\subset{\widetilde{X}}$, we obtain the following from the sequence in the above
\begin{equation}\label{Omega_E}
\left (\Omega^1_{\widetilde{X}}\right)_{|E}\cong\Oo_{E}(1)\oplus\Oo_{E}(-2).
\end{equation}

\noindent For an arrangement $\Dd=\{C_1, C_2, \ldots, C_m\}$ of smooth irreducible curves on $X$, we denote by 
\[
\widetilde{\Dd}=\{\widetilde{C}_{1},\widetilde{C}_{2},\ldots,\widetilde{C}_{m}\}
\]
an arrangement of smooth irreducible curves with $\widetilde{C}_i$ the strict transform of $C_i$ along $\eta$ for each $i$. We also set
\[
\widetilde{\Dd}^E=\widetilde{\Dd} \cup \{E_1, \dots, E_k\}.
\]

\begin{theorem}\label{main}
If $Z \cap C_i=\emptyset$ for any $i$, then we have
\[
\eta^{*}\Omega^1_{X}(\log\Dd)\cong\Omega^1_{\widetilde{X}}(\log\widetilde{\Dd}^{E})(-E).
\]
\end{theorem}

\begin{proof}
Tensoring the Poincar\'e residue sequences for $\Omega^1_{\widetilde{X}}(\log\widetilde{\Dd})$ and $\Omega^1_{\widetilde{X}}(\log\widetilde{\Dd}^{E})$ with $\Oo_{\widetilde{E}}$, we get
\[
\Omega^1_{\widetilde{X}}(\log\widetilde{\Dd})_{|E}\cong\Oo_{E}(1)\oplus\Oo_{E}(-2) \;\;\;\;
\text{and}\;\;\;\; \Omega^1_{\widetilde{X}}(\log\widetilde{\Dd}^{E})_{|E}\cong\Oo_{E}(-1)\oplus\Oo_{E}(-1),
\]
using (\ref{Omega_E}). Now, by tensoring the standard exact sequence for $E\subset \widetilde{X}$ 
\[
0\to \Oo_{\widetilde{X}}(-E) \to \Oo_{\widetilde{X}} \to \Oo_{E}\to 0,
\]
by the logarithmic vector bundles $\Omega^1_{\widetilde{X}}(\log\widetilde{\Dd})$ and $\Omega^1_{\widetilde{X}}(\log\widetilde{\Dd}^{E})$, one can get a commutative diagram 
\begin{equation}\label{comm1}
\begin{tikzcd}[
  row sep=small,
  ar symbol/.style = {draw=none,"\textstyle#1" description,sloped},
  isomorphic/.style = {ar symbol={\cong}},
  ]
 & & & & 0\ar[d] & \\
 & & 0\ar[d] &0\ar[d] & \Oo_{E}(1) \ar[d] &  \\
 & 0\ar[r] & \Omega^1_{\widetilde{X}}(\log\widetilde{\Dd})(-E) \ar[d]\ar[r] & \Omega^1_{\widetilde{X}}(\log\widetilde{\Dd}) \ar[d]\ar[r, "f"] & \Omega^1_{\widetilde{X}}(\log\widetilde{\Dd})\tensor\Oo_{E} \ar[d, "~g"]\ar[r] & 0 \\
 & 0\ar[r] & \Omega^1_{\widetilde{X}}(\log\widetilde{\Dd}^{E})(-E) \ar[d]\ar[r] & \Omega^1_{\widetilde{X}}(\log\widetilde{\Dd}^{E}) \ar[d]\ar[r] & \Omega^1_{\widetilde{X}}(\log\widetilde{\Dd}^{E})\tensor\Oo_{E} \ar[d]\ar[r] & 0. \\
 & &\Oo_{E}(1)\ar[d] & \Oo_{E}\ar[d] & \Oo_{E} \ar[d] \\
 & & 0 & 0 & 0. & 
\end{tikzcd}
\end{equation}
Here, the last vertical exact sequence is obtained by tensoring the middle vertical sequence by $\Oo_E$, because $Tor_{\widetilde{X}}^1(\Oo_E, \Oo_E) \cong \Oo_E(1)$.
Since the map $g$ in (\ref{comm1}) factors through the line bundle $\Oo_{E}(-2)$,
\[
\begin{tikzcd}[
  row sep=small,column sep=small,
  ar symbol/.style = {draw=none,"\textstyle#1" description,sloped},
  isomorphic/.style = {ar symbol={\cong}},
  ]
 0\ar[r] & \Oo_{E}(1)\ar[r] & \Omega^1_{\widetilde{X}}(\log\widetilde{\Dd})\tensor\Oo_{E} \ar[d, isomorphic]\ar[rr] & & \Omega^1_{\widetilde{X}}(\log\widetilde{\Dd}^{E})\tensor\Oo_{E} \ar[d, isomorphic]\ar[r] & \Oo_{E} \ar[r] & 0. \\
 & & \Oo_{E}(1)\oplus\Oo_{E}(-2)\ar[dr, two heads] & & \Oo_{E}(-1)\oplus\Oo_{E}(-1) & & \\
 & & &\Oo_{E}(-2)\ar[ur, hook] & & &
\end{tikzcd}
\]
so that $\mathrm{Im}(g)=\Oo_E(-2)$, we get another commutative diagram
\begin{equation}\label{comm2}
\begin{tikzcd}[
  row sep=small,
  ar symbol/.style = {draw=none,"\textstyle#1" description,sloped},
  isomorphic/.style = {ar symbol={\cong}},
  ]
  & &0\ar[d] & 0 \ar[d] &  \\
  & & \Omega^1_{\widetilde{X}}(\log\widetilde{\Dd}) \ar[d]\ar[r, "g\circ f"] & \Oo_{E}(-2) \ar[d]\ar[r] & 0 \\
  0\ar[r] & \Omega^1_{\widetilde{X}}(\log\widetilde{\Dd}^{E})(-E) \ar[r] & \Omega^1_{\widetilde{X}}(\log\widetilde{\Dd}^{E}) \ar[d]\ar[r] & \Omega^1_{\widetilde{X}}(\log\widetilde{\Dd}^{E})\tensor\Oo_{E} \ar[d]\ar[r] & 0. \\
  & & \Oo_{E}\ar[d]\ar[r, isomorphic] & \Oo_{E} \ar[d] \\
  & & 0 & 0 & 
\end{tikzcd}
\end{equation}
Then, by the Snake Lemma, we have $\ker(g \circ f)\cong\Omega^1_{\widetilde{X}}(\log\widetilde{\Dd}^{E})(-E)$. On the other hand, since $Z\intsec C_i=\emptyset$, we have that $\widetilde{C_i}\cap E_i=\emptyset$ for any $i$. Thus by \cite[Proposition 3.2.(iv)]{deligne2006equations}, there exist a morphism 
\[
\psi : \quad \eta^{*}\Omega^1_{X}(\log\Dd) \longrightarrow \Omega^1_{\widetilde{X}}(\log\widetilde{\Dd}).
\]
In fact, this morphism is injective because $\eta^{*}\Omega^1_{X}(\log\Dd)$ is locally free and $\psi$ has full rank. Setting $\Qq :=\mathrm{coker}(\psi)$, we have a commutative diagram 
\begin{equation}\label{commm3}
\begin{tikzcd}[
  row sep=small,
  ar symbol/.style = {draw=none,"\textstyle#1" description,sloped},
  isomorphic/.style = {ar symbol={\cong}},
  ]
 & & 0\ar[d] &0\ar[d] & 0 \ar[d] &  \\
 & 0\ar[r] & \eta^{*}\Omega^1_{X} \ar[d]\ar[r] & \Omega^1_{\widetilde{X}} \ar[d]\ar[r] & \Oo_{E}(-2) \ar[d, equal]\ar[r] & 0 \\
 & 0\ar[r] & \eta^{*}\Omega^1_{X}(\log\Dd) \ar[d]\ar[r] & \Omega^1_{\widetilde{X}}(\log\widetilde{\Dd}) \ar[d]\ar[r,"\phi"] & \Qq \ar[r] & 0. \\
 & &\bigoplus_{i=1}^{m}\eta^{*}\Oo_{C_{i}}\ar[d]\ar[r, isomorphic] & \bigoplus_{i=1}^{m}\Oo_{\widetilde{C}_i}\ar[d] & \\
 & & 0 & 0 
\end{tikzcd}
\end{equation} 
Note that any non-trivial map from $\Omega_{\widetilde{X}}^1(\log \widetilde{\Dd})$ to $\Oo_E(-2)$ factors through 
\[
\left(\Omega_{\widetilde{X}}^1(\log \widetilde{\Dd})\right)_{|E} \cong \Oo_E(1)\oplus \Oo_E(-2),
\]
and so it must be unique up to constant. Thus we get $g \circ f=\phi$ and so 
\[
\eta^{*}\Omega^1_{X}(\log\Dd)\cong\ker(\phi) \cong\Omega^1_{\widetilde{X}}(\log\widetilde{\Dd}^{E})(-E). \qedhere
\]
\end{proof}

\begin{corollary}\label{maincor}
If $\Dd$ is a Torelli arrangement on $X$ with $Z \cap C_i=\emptyset$, then so are $\widetilde{\Dd}$ and $\widetilde{\Dd}^E$ on $\widetilde{X}$.
\end{corollary}

\begin{proof}
Applying the push-forward functor $\eta_*$ to the second horizontal sequence of (\ref{commm3}), one can get an exact sequence
\[
0\to \Omega_X^1(\log \Dd) \cong \eta_* \eta^* \Omega_X^1(\log \Dd) \to \eta_* \Omega_{\widetilde{X}}^1(\log \widetilde{\Dd}) \to \eta_* \Oo_E(-2) \cong 0,
\]
by the projection formula together with $\eta_* \Oo_{\widetilde{X}}\cong \Oo_X$. The last vanishing comes from the Theorem on Formal Functions in \cite[Section III.11]{hartshorne2013algebraic}. Thus we get $\eta_*\Omega_{\widetilde{X}}^1(\log \widetilde{\Dd}) \cong \Omega_X^1(\log \Dd)$ and so $\widetilde{\Dd}$ is a Torelli arrangement. On the other hand, by Theorem \ref{main} we get
\[
\Omega_X^1(\log \Dd) \cong \eta_* \left(\Omega_{\widetilde{X}}^1(\log \widetilde{\Dd}^E)(-E)\right)
\]
and so $\widetilde{\Dd}^E$ is also a Torelli arrangement. 
\end{proof}

\begin{example}\label{deg1}
Let $\eta:\FF=\mathrm{Bl}_{p}\PP^2\rightarrow\PP^2$ be the blow-up of $\PP^2$ at a point $p\in \PP^2$ with the exceptional divisor $E$. For an arrangement $\Ll=\left\{L_1,L_2,\ldots,L_m\right\}$ of $m$ distinct lines in $\PP^2$ with simple normal crossings, we let $\widetilde{L}_i$ be the strict transform of $L_i$ for each $i$ to consider an arrangement $\widetilde{\Ll}=\{\widetilde{L}_1, \dots, \widetilde{L}_m\}$. Assume now that $p\not\in L_i$ for each $i$ so that $\widetilde{L}_i\in |\eta^*\Oo_{\PP^2}(1)|$. Setting $\widetilde{\Ll}^{E}:=\widetilde{\Ll}\cup\{E\}$ as before, we have
\[
\Omega^1_{\FF}(\log \widetilde{\Ll}^{E})\cong\left(\eta^*\Omega^1_{\PP^2}(\log\Ll)\right)(E)
\]
by Theorem \ref{main}. In particular, by \cite[Proposition 2.10]{dolgachev1993arrangements}, we obtain the follwing for $m\le 3$:
\[
\Omega^1_{\FF}(\log \widetilde{\Ll}^{E})\cong \left\{
\begin{array}{ll}
  \Oo_{\FF}(-H+E)^{\oplus 2} &\hbox{   ($m=1$)}\vspace{1mm}\\
  \Oo_{\FF}(E)\oplus\Oo_{\FF}(-H+E) &\hbox{    ($m=2$)}\vspace{1mm}\\
  \Oo_{\FF}(E)^{\oplus 2} &\hbox{    ($m=3$)},
\end{array}
\right.
\]
where $\Oo_{\FF}(H)\cong \eta^{*}\Oo_{\PP^2}(1)$. 
%This cases will be again recalled in the next section, considering $\FF$ as a rational ruled surface.
For $m=4$ we have 
\[
\Omega^1_{\FF}(\log\widetilde{\Ll}^{E})\cong\eta^{*}\left(\Tt_{\PP^2}(-1)\right)(E).
\]
On the other hand, for $m\geq 5$, by Corollary \ref{maincor} and \cite{dolgachev1993arrangements,valles2000nombre} the arrangements $\widetilde{\Ll}$ and $\widetilde{\Ll}^{E}$ are Torelli, unless $L_i$'s osculate any smooth conic in $\PP^2$. Note that in the case $m=5$, it cannot be a Torelli arrangement because the given five lines in general position always determine a conic that mutually tangent to the five lines. 
\end{example}

\begin{example}
Let $\widetilde{C}\in\left|\eta^{*}\Oo_{\PP^2}(2)\right|$ be a smooth irreducible curve which corresponds to a smooth conic $C$ in $\PP^2$ and set $\widetilde{\Cc}=\{\widetilde{C}\}$. It was shown in \cite{angelini2014logarithmic} that $\Omega^1_{\PP^2}(\log C)$ is isomorphic to $\Tt_{\PP^2}(-2)\cong\Omega^1_{\PP^2}(1)$. Thus by Theorem \ref{main} we get $\Omega^1_{\FF}(\log\widetilde{\Cc}^{E})\cong \left(\eta^{*}\Omega^1_{\PP^2}\right)(H+E)$.
\end{example}

Now let $\Cc=\{C_1,\ldots,C_m\}$ be an arrangement of $m$ smooth irreducible curves with simple normal crossings in $\PP^2$, where $C_i\in|\Oo_{\PP^2}(d_i)|$ for each $i$. We will show that if $\Omega^1_{\PP^2}(\log\Cc)$ is stable, then $\Omega^1_{\FF}(\log\widetilde{\Cc}^{E})$ is an element of the moduli space $\mathbf{M}_{\FF}(c_1,c_2)$ of stable sheaves of rank two with its Chern classes
\[
    c_1=\left(-3+\sum_{i=1}^{m}d_i\right)H+2E \quad\text{and}\quad c_2=\sum_{i=1}^{m}{d_i}^2+\sum_{i<j}d_{i}d_{j}-3\sum_{i=1}^{m}d_i +2
\]
with respect to the anticanonical polarization $-K_{\FF}=3H-E$. After then we discuss the generic injectivity of a map
\[
\Phi:\quad \mathbf{F} \longrightarrow \mathbf{M}_{\FF}(c_1, c_2),
\]
where $\mathbf{F}$ denotes the family of arrangements $\widetilde{\Cc}^{E}$ in $\FF$. We note that the Chern classes of $\Omega^1_{\PP^2}(\log\Cc)$ can be obtained by using the resolution of \cite[Theorem 4.3] {angelini2014logarithmic} for $n=2$ :
\begin{equation}\label{Ancona}
0 \to \Dirsum_{i=1}^{m}\Oo_{\PP^2}(-d_i)\to \Oo_{\PP^2}^{\dirsum(m-1)}\dirsum\Oo_{\PP^2}(-1)^{\dirsum 3}\to \Omega^1_{\PP^2}(\log\Cc)\to 0,    
\end{equation}
that is 
\[
   c_1(\Omega^1_{\PP^2}(\log\Cc))=\sum_{i=1}^{m}d_i-3, \quad\text{and}\quad c_2(\Omega^1_{\PP^2}(\log\Cc))=\sum_{i=1}^{m}{d_i}^2+\sum_{i<j}d_{i}d_{j}-3\sum_{i=1}^{m}d_i +3
\]
and then 
\[
c_1(\Omega^1_{\FF}(\log\widetilde{\Cc}^{E})=\eta^{*}c_1(\Omega^1_{\PP^2}(\log\Cc))+2E \;\;\;\;\text{and}\;\;\;\; c_2(\Omega^1_{\FF}(\log\widetilde{\Cc}^{E})
=\eta^{*}c_2(\Omega^1_{\PP^2}(\log\Cc))-1.
\]

\begin{proposition}\label{gstab}
 If $\Omega^1_{\PP^2}(\log\Cc)$ is stable then $\Omega^1_{\FF}(\log\widetilde{\Cc}^{E})$ is $(-K_{\FF})$-stable. 
\end{proposition}
\begin{proof}
Assuming that $\Omega^1_{\PP^2}(\log\Cc)$ is stable, suppose that $\Omega^1_{\FF}(\log\widetilde{\Cc}^{E})$ is not $(-K_{\FF})$-stable and let $\Oo_{\FF}(aH+bE)$ be a destabilizing line bundle
\begin{equation}\label{destab}
\Oo_{\FF}(aH+bE) \hookrightarrow \Omega^1_{\FF}(\log\widetilde{\Cc}^{E}).
\end{equation}
Setting $\delta:=\sum_{i=1}^{m}d_i$, we should have
\[
\mu(\Oo_{\FF}(aH+bE)) = 3a+b \ge \frac{3\delta-7}{2} = \mu\left(\Omega^1_{\FF}(\log\widetilde{\Cc}^{E})\right).
\]
Applying the push-forward functor $\eta_{*}$ to \eqref{destab}, we get
\[
\Ii_{Z, \PP^2}(a) \hookrightarrow \eta_{*}\Omega^1_{\FF}(\log\widetilde{\Cc}^{E})\cong\Omega^1_{\PP^2}(\log\Cc)
\]
for some zero-dimensional subscheme $Z$ supporting $p$; possibly $Z=\emptyset$. Since $\Omega^1_{\PP^2}(\log\Cc)$ is stable, we get $2a\leq\delta-4$. On the other hand, we can take a general fibre $f$ in $\FF$ such that the splitting of $\Omega_{\PP^2}^1(\log \Cc)$ is generic over the line $\eta(f)=L$. To be explicit, 
\[
\left(\Omega^1_{\PP^2}(\log\Cc)\right)_{|L}\cong\left\{
\begin{array}{ll}
  \Oo_{L}\left(\frac{\delta-3}{2}\right)^{\oplus 2} &\hbox{($\delta$ : odd)}\vspace{2mm}\\
  \Oo_{L}\left(\frac{\delta-2}{2}\right)\oplus\Oo_{L}\left(\frac{\delta-4}{2}\right) &\hbox{($\delta$ : even)}
\end{array}
\right.
\]
By \cite[Remark III.9.3.1]{hartshorne2013algebraic}, we have a natural map 
\[
\left(\Omega^1_{\PP^2}(\log\Cc)\right)_{|L} \longrightarrow \eta_* \left(\Omega^1_{\FF}(\log\widetilde{\Cc}^{E})_{|f}\right). 
\]
Clearly, this map has full rank and so it is injective. Thus we get
\[
\Omega^1_{\FF}(\log\widetilde{\Cc}^{E})_{|f}\cong\left\{
\begin{array}{ll}
  \Oo_{f}\left(\frac{\delta-1}{2}\right)^{\oplus 2} &\hbox{($\delta$ : odd)}\vspace{2mm}\\
  \Oo_{f}\left(\frac{\delta}{2}\right)\oplus\Oo_{f}\left(\frac{\delta-2}{2}\right) &\hbox{($\delta$ : even)}
\end{array}
\right.
\]
for a generic fibre $f$. In particular, the restriction of the map \eqref{destab} to $f$ gives $a+b\leq \frac{\delta}{2}$ in either case. Consequently, we have
\[
3a+b=2a+(a+b)\leq\frac{3\delta-8}{2}<\frac{3\delta-7}{2},
\]
a contradiction.
\end{proof}

For the stability of $\Omega^1_{\PP^2}(\log\Cc)$, we have the following lemma. 

\begin{lemma}
Let $\Cc=\{C_1,\ldots,C_m\}$ be an arrangement of $m$ smooth irreducible curves with simple normal crossings and $C_i\in |\Oo_{\PP^2}(d_i)|$ for each $i$. We may assume that $d_1\geq d_2\geq\dots\geq d_m$ without loss of generality. Then $\Omega^1_{\PP^2}(\log\Cc)$ is stable if and only if the corresponding $(m+1)$-tuple $(m; d_1, \dots, d_m)$ is not contained in the set
\[
\mathbb{S}=\left\{ ~(1;1) ,~(2;1,1) ,~ (2;2,1) ,~(3;1,1,1)~\right\}.
\]
\end{lemma}

\begin{proof}
    Let $\Ff:=\Omega^1_{\PP^2}(\log\Cc)(k)$ be a \textit{normalized bundle} of $\Omega^1_{\PP^2}(\log\Cc)$ with 

    \[
    k = \begin{cases}
         \quad-\cfrac{c_1(\Omega^1_{\PP^2}(\log\Cc))}{2} & \quad\text{if $c_1(\Omega^1_{\PP^2}(\log\Cc))$ is even}\vspace{2mm} \\ 
         \quad-\cfrac{c_1(\Omega^1_{\PP^2}(\log\Cc))+1}{2} & \quad\text{if $c_1(\Omega^1_{\PP^2}(\log\Cc))$ is odd}.
    \end{cases}
    \]
Then using the sequence \eqref{Ancona}, we get a long exact sequence of cohomology
\[
0 \to \mathrm{H}^0\left(\Dirsum_{i=1}^{m}\Oo_{\PP^2}(-d_{i}+k)\right) \to \mathrm{H}^0\left(\Oo_{\PP^2}(k)^{\dirsum(m-1)}\dirsum\Oo_{\PP^2}(-1+k)^{\dirsum 3}\right) \to \mathrm{H}^0\left(\Ff\right) \to 0
\]
Note that the first term $\mathrm{H}^0(\Dirsum_{i=1}^{m}\Oo_{\PP^2}(-d_{i}+k))$ vanishes except the case $(1;1)$, while the midterm $\mathrm{H}^0(\Oo_{\PP^2}(k)^{\dirsum(m-1)})\dirsum\mathrm{H}^0( \Oo_{\PP^2}(-1+k)^{\dirsum 3})$ vanishes except all the cases in $\mathbb{S}$. Conversely, $\mathrm{H}^0(\Ff)$ does not vanish for all the cases in $\mathbb{S}$. Therefore, the arrangement $\Cc$ does not correspond to one of the cases in $\mathbb{S}$ if and only if $\mathrm{H}^0(\Ff)=0$, which is again equivalent to the stability of $\Omega^1_{\PP^2}(\log\Cc)$ by \cite[Ch.2, Lemma 1.2.5]{okonek1980vector}. Moreover, we remark that $\Omega^1_{\PP^2}(\log\Cc)$ is semistable for the case $(2;2,1)$, while the bundle is a direct sum of line bundles in the other cases of $\mathbb{S}$.
%{\color{blue}
%(Re-state) 
%Thus we have
%\[
%\mathrm{h}^0(\Ff)=(m-1)\binom{k+2}{2}+3\binom{k+1}{2}-\sum_{i=1}^{m}\binom{-d_i+k+2}{2}. 
%\]
%Note that, when $m\geq 4$ we always have $k<0$ so that $h^0(\Ff)=0$. Now we assume that $m<4$. If $m=1$, then $\mathrm{h}^0(\Ff)$ vanishes precisely when $d_1\geq 2$, which yield $k\leq 0$. In fact, when $d_1=1$ we have $k=1$ and $\mathrm{h}^0(\Ff)=2$. Similarly, we can easily check that for the cases $m=2$ and $m=3$, $\mathrm{h}^0(\Ff)$ does not vanishes if and only if $\Cc$ correspondent to one of the cases in $S$. Therefore, we have the conclusion from \cite[Ch.2, Lemma 1.2.5]{okonek1980vector}.
%}
\end{proof}

 %\begin{corollary}{\label{maincor3}}
 %   Setting $\mathbf{F}$ the family of arrangements $\widetilde{\Ll}^{E}$ for $m\geq 6$ such that $\Ll$ do not osculate any smooth conic, we get a generically injective map 
 %   \[
 %   \Phi:\quad \mathbf{F} \longrightarrow \mathbf{M}_{\FF}(c_1, c_2),
 %   \]
 %   where $\mathbf{M}_{\FF}(c_1, c_2)$ is the moduli space of stable sheaves of rank two with Chern classes 
 %   \[
 %   c_1=(m-3)H+2E, \quad c_2=\frac{(m-1)(m-4)}{2}
 %   \]
 %   with respect to the anticanonical polarization $-K_{\FF}=3H-E$.
%\end{corollary}
%\begin{proof}
%Setting $d_i=1$ for all $i$, by Corollary \ref{maincor}, Proposition \ref{gstab} and \cite[Corollary 3.1]{valles2000nombre}.
%\end{proof}

\begin{corollary}
Let $\Cc=\{C_1,\ldots,C_m\}$ be an arrangement of $m$ smooth irreducible curves with simple normal crossings and $C_i\in |\Oo_{\PP^2}(d)|$ for each $i$. Set $\mathbf{F}$ the family of arrangements $\widetilde{\Cc}^{E}$ and assume that $m\geq \binom{2+d}{d}+3$. Then we get a generically injective map
\[
    \Phi:\quad \mathbf{F} \longrightarrow \mathbf{M}_{\FF}(c_1, c_2),
    \]
    where $\mathbf{M}_{\FF}(c_1, c_2)$ is the moduli space of stable sheaves of rank two with Chern classes 
    \[
    c_1=(md-3)H+2E \quad\text{and}\quad c_2=\frac{d^2m^2+d(d-6)m+4}{2}
    \]
    with respect to the anticanonical polarization $-K_{\FF}=3H-E$. 
\end{corollary}
\begin{proof}
The generic injectivity comes from \cite[Corollary 3.1]{valles2000nombre} for $d=1$, and \cite[Corollary 5.6]{angelini2014logarithmic} for $d\geq 2$. Then the assertion follows from Corollary \ref{maincor}.
\end{proof}

\begin{remark}
In \cite{HMPV} a general philosophy about the Torelli property on logarithmic vector bundle was suggested: one can get a positive answer to the Torelli property, after sufficiently many blow-ups. This point of view can be reformulated as follows, using Theorem \ref{main}. Let $\mathbf{F}$ be the set of arrangements $\Dd'$ of divisors in the system $|\Oo_X(\Dd)|$ with $\Omega_X^1(\log \Dd') \cong \Omega_X^1(\log \Dd)$. Setting $\eta : \widetilde{X} \rightarrow X$ the blow-up of $X$ along a set $Z=\{p_1, \dots, p_k\}$ of $k$ points in $X$, we get an exact sequence
\[
0\to \Omega_{\widetilde{X}}^1(\log \widetilde{\Dd'}) \to \Omega_{\widetilde{X}}^1(\log \widetilde{\Dd'}^E) \cong \eta^*\Omega_X^1(\log \Dd')(E) \cong \eta^*\Omega_X^1(\log \Dd)(E)\to \Oo_E \to 0
\]
with $E=E_1+\cdots+ E_k$. In particular, the isomorphism class of $\Omega_{\widetilde{X}}^1(\log \widetilde{\Dd'})$
is determined by the surjection from $\eta^*\Omega_X^1(\log \Dd)(E)$ to $\Oo_E$, and so it corresponds to a point in 
\[
\underbrace{\PP^1 \times \dots \times \PP^1}_{k \text{ copies}}\cong \prod_{i=1}^{k} \PP \mathrm{Hom}(\eta^*\Omega_X^1(\log \Dd)(E)_{p_i}, \Oo_{p_i}).
\]
Hence, we get a map 
\[
\Psi :\quad  \mathbf{F} \dashrightarrow \underbrace{\PP^1 \times \dots \times \PP^1}_{k \text{ copies}},
\]
and the general philosophy can be rephrased as follows: the map $\Psi$ can be expected to be generically injective for general $Z$ with sufficiently large $k$. 
\end{remark}

%%%%%%%%%%%%%%%%%%%%%%%%%%%%%%%%%%%
\section{Hirzebruch surfaces}
In this section we will focus on the blown-up plane $\FF=\mathrm{Bl}_{p}\PP^{2}$ of Example \ref{deg1} and consider it as a rational ruled surface, namely the Hirzebruch surface. From this view, we will introduce a methodology to find the canonical exact sequence for a logarithmic vector bundle on $\FF$, introduced in \cite{brosius1983rank}.

Let $\FF_e=\PP(\Oo _{\PP^1}\oplus \Oo _{\PP^1}(-e))$ for $e\ge 0$ be the Hirzebruch surface with minimal self-intersection number $-e$ and its ruling
$\pi : \FF_e \rightarrow \PP^1$. We have $\mathrm{h}^1(\Oo_{\FF_e})=g(\PP^1)=0$, and for $e>0$ we have 
\[
\mathrm{h}^0(\Tt_{\FF_e}) = \dim \mathrm{Aut}(\FF_e) =\dim \mathrm{Aut}(\PP^1) +\dim \End_{\PP^1} (\Oo _{\PP^1}\oplus \Oo _{\PP^1}(-e))-1 =e+5. 
\]
We also have $\mathrm{Pic}(\FF_e)\cong \ZZ^{\oplus 2}\cong \ZZ \langle h,f\rangle$, where $f$ is a fiber of a ruling of $\FF_e$ and $h$ is the section with $h^2=-e$, $h\cdot f=1$ and $f^2=0$.

\begin{remark}
It is well known that for a divisor $D=ah+bf$ a line bundle $\Ll=\Oo_{\FF_e}(D)$ on $\FF_e$ is very ample if and only if it is ample, which is also equivalent to the condition that $a>0$ and $b>ae$. The line bundle $\Ll$ is effective if and only if $a\geq0$ and $b\geq0$. There is also a similar equivalent condition for the linear system $|\Ll|$ to contain an irreducible (smooth) curve, that is either 
\[
\mathrm{(i)}~ D=h, \quad \mathrm{(ii)}~D=f, \quad \mathrm{(iii)}~a>0,~ b>ae,\quad  \text{or} \quad  \mathrm{(iv)}~e>0, ~a>0,~ b=ae.
\]
\end{remark}

Since $\omega _{\FF_e}\cong \Oo _{\FF_e}(-2h-(e+2)f)$, there is a subcanonical polarization of $\FF_e$ if and only if $e=0,1$. Since the ruling $\pi : \FF_e\rightarrow \PP^1$ is a submersion, it induces a surjective map $ \Tt_{\FF_e} \rightarrow  \pi^\ast \Tt_{\PP^1} \cong \Oo _{\FF_e}(2f)$. So from $\omega _{\FF_e}^\vee \cong \Oo _{\FF_e}(2h+(e+2)f)$ we get that $\Tt_{\FF_e}$ fits into an exact sequence
\begin{equation}\label{eqy1}
0 \to \Oo _{\FF_e}(2h+ef) \to \Tt_{\FF_e} \to \Oo _{\FF_e}(2f)\to 0.
\end{equation}

\begin{remark}\label{rmk31}
The Chern classes of $\Tt_{\FF_e}$ is $(c_1, c_2)=(2h+(e+2)f, 4)$. By restricting the sequence (\ref{eqy1}) to $f$, we get that the splitting of $\Tt_{\Ff_e}$ over each fibre $f$ is
\[
\left(\Tt_{\FF_e}\right)_{|f}\cong \Oo_f(2)\oplus \Oo_f.
\]
On the other hand, from the normal exact sequence for $h\subset \FF_e$, we get 
\[
\left(\Tt_{\FF_e}\right)_{|h}\cong \Oo_h(2)\oplus \Oo_h(-e).
\]
\end{remark}

\begin{remark}
In the case $e=1$ the surface $\FF:=\FF_1$ is isomorphic to the blow-up of $\PP^2$ at one point $p$ with $\eta: \FF\cong \mathrm{Bl}_p(\PP^2) \rightarrow \PP^2$. Recall that we already used this notation for the single point blown-up of $\PP^2$ in the previous section. The surface $\FF$ is the del Pezzo surface of degree $8$, and so $\mathrm{h}^0(\Tt_{\FF})=6$. From this point of view, note that a system of curves $\left|\eta^{*}\Oo_{\PP^2}(1)\right|$ is equal to the system $\left|\Oo_{\FF}(h+f)\right|$.
\end{remark}

Let $\Dd=\{D_1, D_2, \ldots, D_m\}$ be an arrangement of smooth irreducible curves  $D_i\in |\Oo_{\FF}(a_ih+b_if)|$ with simple normal crossings on $\FF$ so that $\Dd\in |\Oo_{\FF}(\alpha h+ \beta f)|$ with $\alpha=\sum_{i=1}^{m}a_i$ and $\beta=\sum_{j=1}^{m}b_j$. Then from the Poincar\'e residue sequence for $\Omega^1_\FF(\log\Dd)$ we get
\[
c_1(\Omega^1_\FF(\log\Dd))=(\alpha-2)h+(\beta-3)f\]
and 
\[
c_2(\Omega^1_\FF(\log\Dd))=4-\alpha-2\beta+\sum_{i=1}^m(2a_{i}b_{i}-a_{i}^2)+\sum_{i<j}(a_ib_j+a_jb_i-a_ia_j).
\]

\noindent Ever since Brosius showed in \cite{brosius1983rank} that any vector bundle of rank two on ruled surface can be described by a canonical extension, there have been many studies using similar tools to study vector bundles of rank two on Hirzebruch surface; see \cite{aprodu1996existence} and \cite{brinzanescu1991algebraic}. Throughout this section, we use notations and terminologies of the survey article \cite{aprodu2012rank}. The following result is one of main methodologies that we keep using throughout this section; see \cite{aprodu2012rank}. 

\begin{proposition}\label{AB}
Let $\Ff$ be a vector bundle of rank two on $\FF$. 
\begin{itemize}
\item [(i)] Set $(d, d')\in \ZZ^{\oplus 2}$ with $d\ge d'$ the splitting type of $\Ff$ over a generic fiber $f$. Then $\pi_* \Ff(-dh)$ is a vector bundle of rank $1$ (resp. rank $2$) if $d>d'$ (resp. $d=d'$). 
\item [(ii)] If $d>d'$, set $r:=\deg\pi_{*}\Ff(-dh)$. If $d=d'$, set $\pi_{*}\Ff(-dh)\cong\Oo_{\PP^1}(r)\dirsum\Oo_{\PP^1}(s)$ with $r\geq s$. Then we have
\[
-r=\inf \left\{k \in \ZZ \;\big|\;\exists \Ll\in \Pic^{k}(\PP^1) \;\textrm{with }\;  \mathrm{h}^0(\Ff(-dh)\tensor\pi^{*}\Ll)\ne 0 \right\}.
\]
\end{itemize}
\end{proposition}

\noindent In Proposition \ref{AB} we may associate to $\Ff$ a pair of integers $(d,r)$ for which there exists an exact sequence
\begin{equation}\label{eqa111}
0 \to \Oo_{\FF}(dh+rf) \to \Ff \to \Ii_{Z, \FF}(d'h+r'f) \to 0
\end{equation}
for some zero-dimensional subscheme $Z\subset \FF$. We call it a {\it canonical extension} of $\Ff$. We have
\[
\deg (Z) = c_2(\Ff)+a(d-r)-bd+2dr-d^2,
\]
together with $c_1(\Ff)=ah+bf$ and $r+r'=b$. We notice that 
\[
\dim\Ext^1_{\FF}(\Ii_{Z,\FF}(d'h+r'f), \Oo_{\FF}(dh+rf)) = \deg (Z)+\mathrm{h}^1(\Oo_{\FF}((d-d')h+(r-r')f)).
\]

\begin{remark}\label{pi}
In the sequence \eqref{eqa111}, the length of $Z$ measures the \textit{uniformity} of $\Ff$. That is, $\Ff$ has the same splitting type over all the fibers of $\pi$ if and only if $Z=\emptyset$. We say that such a bundle is \textit{$\pi$-uniform}; see \cite{brinzanescu1991algebraic} and \cite{ishimura1979pi}. For example, as we shall see in Example \ref{ex_f}, the logarithmic bundle $\Omega_{\FF}^1(\log f)$ associated to a fibre $f\subset \FF$ is $\pi$-uniform with the splitting type $(0,-2)$, and it admits a canonical extension (\ref{seqf2}). 
\end{remark}

\begin{remark}
As studied in \cite{brosius1983rank} and \cite{brinzanescu1991algebraic}, any vector bundle $\Ff$ of rank two on $\FF$ is described canonically the extension (\ref{eqa111}). If we call the twisted bundle $\Ff(-dh-rf)$ the \textit{normalization} of $\Ff$, then the extension $\eqref{eqa111}$ is unique if and only if the space of global sections of the normalization of $\Ff$ is a 1-dimensional. Thus the extension $\eqref{eqa111}$ is unique if $d>d'$, or $d=d'$ and $r>s$. 
%{\color{blue}On the other hand, if $d=d'$, there is a natural exact sequence
%\[
%0 \to \pi^{*}\pi_{*}\Ff(-dh) \to \Ff(-dh) \to \Qq(-dh) \to 0,
%\]
%where the first map is the natural map, and $\Qq(-dh)$ is a torsion free sheaf supported on the fibres passing through any point of $Z$. Furthermore, $\Qq(-dh)$ has no global sections.
%;see \cite{brosius1983rank}, \cite{friedman2012algebraic} Ch.6. Proposition 11.
%}
\end{remark}

%{\color{blue}
%Before introducing examples, here are a few formulas on $\FF$ for some calculations later.
%}

%{\color{blue}
%\begin{remark}
%Since the Hirzebruch surface $\FF_{e}$ is a ruled surface over $\PP^1$, we have 
%$$\mathrm{H}^{0}(\FF_e, \Oo_{\FF_{e}}(ah+bf))\cong \mathrm{H}^{0}(\PP^1, \pi_{*}\Oo_{\FF_e}(ah+bf)).$$ 
%On the other hand, we note that
%\[
%\pi_{*}\Oo_{\FF_{e}}(ah+bf)\cong\Oo_{\PP^1}(b)\tensor \mathrm{Sym}^{a}(\Oo_{\PP^1}\oplus\Oo_{\PP^1}(-e))\cong
%	   \left
%	   \{
%		\begin{array}{ll}
%			\Dirsum_{k=0}^{a}\Oo_{\PP^1}(b-ke) &\hbox{($a\geq 0$)}\\
%			0 &\hbox{($a<0$)}
%		\end{array}
%		\right.
%\]
%and
%\[
%\mathbf{R}^1\pi_{*}\Oo_{\FF_{e}}(ah+bf)\cong (\pi_{*}\Oo_{\FF_{e}}((-a-2)h))^{\vee}\tensor\Oo_{\PP^1}(e+b)\cong \left\{
%\begin{array}{ll}
%	0 &\hbox{($a>-2$)}\\
%	 \Dirsum_{k=0}^{-a-2}\Oo_{\PP^1}((k+1)e+b)) &\hbox{($a\leq -2$)}.
%\end{array}
%\right.
%\]
%With the Leray spectral sequence for $\pi$ and a line bundle $\Oo_{\FF}(ah+bf)$, we also have that
%\[
%\mathrm{H}^1(\FF_{e}, \Oo_\FF(ah+bf))\cong \left\{
%\begin{array}{ll}
%	\mathrm{H}^0(\PP^1, \Dirsum_{k=1}^{-a-1}\Oo_{\PP^1}(ke+b)) &\hbox{($a\leq -2$)}\vspace{1mm}\\
%	0 &\hbox{($a=-1$)}\vspace{1mm}\\
%	\mathrm{H}^0(\PP^1, \Dirsum_{k=0}^{a}\Oo_{\PP^1}(ke-b-2)) &\hbox{($a\geq 0$); }
%\end{array}
%\right.
%\]
%see \cite{fulger2011some} for the details.
%\end{remark}
%}
Let us consider logarithmic vector bundles attached to some arrangements of lower degree curves on $\FF$, using its canonical extension.    

\begin{example}\label{ex_f}
Set $\Dd=\{f\}$ for a fibre of $\pi : \FF \rightarrow \PP^1$. We have an exact sequence
\begin{equation}\label{seqf1}
0\to \Omega_{\FF}^1 \to \Omega_{\FF}^1(\log f) \to \Oo_{f} \to 0
\end{equation}
and the Chern classes of $\Omega_{\FF}^1(\log f)$ is $(c_1, c_2)=(-2h-2f, 2)$. Then we have $\left(\Omega^1_\FF(\log f)\right)_{|g} \cong \left( \Omega_{\FF}^1 \right)_{|g} \cong  \Oo_g\oplus \Oo_g(-2)$ for any fibre $g\ne f$. Now
tensor the sequence (\ref{seqf1}) with $\Oo_{f}$ to obtain
\[
0\to \mathcal{T}or_{\FF}^1(\Oo_f, \Oo_f) \to \left(\Omega_{\FF}^1\right)_{|f} \to \left(\Omega_{\FF}^1(\log f)\right)_{|f} \to \Oo_f \to 0. 
\]
From $\mathcal{T}or_{\FF}^1(\Oo_f, \Oo_f) \cong \Oo_f$ and Remark \ref{rmk31}, the restriction $\left(\Omega_{\FF}^1(\log f)\right)_{|f}$ fits into the following exact sequence
\[
0\to \Oo_f(-2) \to \left(\Omega_{\FF_e}^1(\log f)\right)_{|f} \to \Oo_f \to 0. 
\]
This implies that $\left(\Omega_{\FF_e}^1(\log f)\right)_{|f}\cong \Oo_f(a)\oplus \Oo_f(-2-a)$ for $a\in \{-1,0\}$. But the function $\psi : \PP^1 \rightarrow \ZZ_{\ge 0}$ defined by $p\mapsto a-b$, where $\left( \Omega_{\FF}^1(\log f) \right)_{|\pi^{-1}(p)}$ has splitting $(a,b)$ with $a\ge b$, is upper semi-continuous. Thus we have $a=0$ and so the bundle $\Omega_{\FF}^1(\log f)$ is $\pi$-uniform with $(d,r)=(0,-1)$. In particular, we get an exact sequence
\begin{equation}\label{seqf2}
0 \to \Oo_{\FF}(-f) \to \Omega_{\FF}^1(\log f) \to \Oo_{\FF}(-2h-f) \to 0,
\end{equation}
where $\dim\Ext_{\FF}^1(\Oo_{\FF}(-2h-f), \Oo_{\FF}(-f))=\mathrm{h}^1(\Oo_{\FF}(2h))=1$. Now, from the twisted standard exact sequence for $h\subset \FF$, we get a long exact sequence of cohomology 
\[
0\cong \mathrm{H}^1(\FF,\Oo_{\FF}(h))\to \mathrm{H}^1(\FF,\Oo_{\FF}(2h))\cong \CC \stackrel{\alpha}{\longrightarrow} \mathrm{H}^1(h,\Oo_h(-2))\cong \CC \to \mathrm{H}^2(\FF,\Oo_{\FF}(h))\cong 0.
\]
Here, we have $\mathrm{H}^1(\FF, \Oo_{\FF}(2h)) \cong \mathrm{Ext}_{\FF}^1(\Oo_{\FF}(-h), \Oo_{\FF}(h))$ and $\mathrm{H}^1(h, \Oo_h(-2)) \cong \Ext_{h}^1( \Oo_h(1), \Oo_h(-1))$ so that the map $\alpha$ sends an extension 
\begin{equation}\label{371}
0\to \Oo_{\FF}(h)\to \Ff \to \Oo_{\FF}(-h) \to 0
\end{equation}
to its restriction to $h$. The unique nontrivial extension class $\xi\in \Ext_h^1(\Oo_h(1), \Oo_h(-1))$ corresponds to an extension
\begin{equation}\label{eqhh}                   
0\to \Oo_h(-1) \to \Oo_h^{\oplus 2} \to \Oo_h(1) \to 0
\end{equation}
and let $\xi'\in \mathrm{H}^1(\FF,\Oo_{\FF}(2h))$ be an element with $\xi'_{|h}=\xi$. The class $\xi'$ corresponds to an extension (\ref{371}) whose restriction to $h$ is (\ref{eqhh}). Returning to the sequence \eqref{seqf2}, we have 
\[
\left(\Omega^1_\FF(\log f)\tensor\Oo_{\FF}(h+f)\right)_{|h}\cong\left(\Omega^1_\FF(\log f)\right)_{|_h}\cong\Oo_h(1)\dirsum\Oo_h(-1).
\]
In particular, the bundle $\Omega_{\FF}^1(\log f)$ does not correspond to the extension class $\xi'$ so that it corresponds to the trivial extension of \eqref{seqf2}. Therefore, we have $\Omega^1_\FF(\log f)\cong\Oo_{\FF}(-f)\dirsum\Oo_{\FF}(-2h-f)$. 
\end{example}

\noindent Extending the consequence of the above example, we give the following lemma for arrangement of many fibres on $\FF$.
\begin{proposition}\label{ex_fm}
	Let $\Dd=\{f_1, \ldots,f_m\}$ with $m$ distinct fibers of $\pi: \FF \rightarrow \PP^1$. Then we have
\[
\Omega^1_{\FF}(\log \Dd)\cong \Oo_\FF((m-2)f)\dirsum\Oo_\FF(-2h-f).
\]
\end{proposition}

\begin{proof}
Note that the case when $m=1$ is already dealt in Example \ref{ex_f} and so we assume that $m\ge 2$. Tensoring $\Oo_f$ to the exact sequnce
\begin{equation}%\label{seqlogfs}
0\to\Omega^1_\FF\to\Omega^1_\FF(\log\Dd)\to\Dirsum_{i=1}^{m}\Oo_{f_i}\to 0, 
\end{equation}
we get
\begin{equation}\label{eqf2}
0\to\dirsum_{i=1}^m\mathcal{T}or^1(\Oo_{f_i},\Oo_f)\to\Oo_f(-2)\dirsum\Oo_f\to\left(\Omega^1_\FF(\log\Dd)\right)_{|f}\to\left(\Dirsum_{i=1}^m\Oo_{f_i}\right)\tensor\Oo_f\to 0.
\end{equation}
If $f\neq f_i$ for any $i$, the sequence (\ref{eqf2}) yields $$\left(\Omega^1_\FF(\log\Dd)\right)_{|f}\cong\Oo_f(-2)\dirsum\Oo_f.$$ 
On the otherhand, if $f=f_i$ for some $i$, we get an exact sequence
\[
0\to\Oo_f\to\Oo_f(-2)\dirsum\Oo_f\to\left(\Omega^1_\FF(\log\Dd)\right)_{|f}\to\Oo_f\to 0,
\]
so that
\[
0\to\Oo_f(-2)\to\left(\Omega^1_\FF(\log\Dd)\right)_{|f}\to\Oo_f\to 0.
\]
Thus we may take $\left(\Omega^1_\FF(\log\Dd)\right)_{|f}$ among two different types of $(0, -2)$ and $(-1, -1)$. In fact, it must be $(0,-2)$ because of the semicontinuity that we have mentioned in Example \ref{ex_f}. Thus $\Omega_\FF^1(\log\Dd)$ is $\pi$-uniform with 
\[
c_1(\Omega_\FF^1(\log\Dd))=-2h+(m-3)f,\quad c_2(\Omega_\FF^1(\log\Dd))=4-2m,\quad (d,r)=(0,m-2),
\]
and so we have an exact sequence
\begin{equation*}
	0 \to \Oo_{\FF}((m-2)f) \to \Omega_\FF^1(\log\Dd) \to \Oo_{\FF}(-2h-f) \to 0.
\end{equation*}
Then $\Ext^1(\Oo_\FF(-2h-f), \Oo_{\FF}((m-2)f))\cong\mathrm{H}^1(\Oo_\FF(2h+(m-1)f))\cong 0$ for $m\geq 2$, and the assertion follows. 
\end{proof}

\begin{example}\label{ex_h}
Set $\Dd=\{h\}$ for the exceptional divisor $h$. Then by Theorem \ref{main} we get that $\Omega_{\FF}^1(\log h) \cong (\eta^* \Omega_{\PP^2}^1)(h)$. Here, we suggest a different approach to study this logarithmic vector bundle in terms of the canonical exact sequence in Proposition \ref{AB} to obtain Proposition \ref{prop3.11}. For a fibre $f$ of the map $\pi : \FF \rightarrow \PP^1$, we get from \cite[Lemma 2.2]{HMPV} the following commutative diagram
\begin{equation}\label{ddia1}
\begin{tikzcd}[
  row sep=small,
  ar symbol/.style = {draw=none,"\textstyle#1" description,sloped},
  isomorphic/.style = {ar symbol={\cong}},
  ]
&0\ar[d] &0\ar[d]   \\
0\ar[r] &\Oo_f(1) \ar[r] \ar[d] &\Tt_f\cong \Oo_f(2)  \ar[d]\ar[r] & \Oo_q \ar[r]\ar[d, isomorphic] &0\\
0\ar[r] &\left( \Tt_{\FF}(-\log h)\right)_{|f} \ar[r] \ar[d] & \left( \Tt_{\FF}\right)_{|f}  \ar[r] \ar[d] &\Oo_q \ar[r] & 0\\
&\Oo_f \ar[r, isomorphic] \ar[d]&\Nn_{f|\FF}\cong \Oo_f \ar[d]  \\
&0 &0
\end{tikzcd}
\end{equation}
with $q:=h \cap f$, from which we get $\Omega^1_\FF(\log h)|_f \cong \Oo_f\dirsum\Oo_f(-1)$. In particular, the bundle $\Omega_{\FF}^1(\log h)$ is $\pi$-uniform and we get $(d,r)=(0,-2)$ in Remark \ref{pi}, which gives an exact sequence
\begin{equation}\label{extlogh}
0 \to \Oo_{\FF}(-2f) \to \Omega_{\FF}^1(\log h) \to \Oo_{\FF}(-h-f) \to 0. 
\end{equation}
Note that 
\[
\Ext_{\FF}^1(\Oo_{\FF}(-h-f), \Oo_{\FF}(-2f))\cong \mathrm{H}^1(\Oo_{\FF}(h-f))
\]
is $1$-dimensional. From the twisted standard sequence for $h\subset\FF$, we get a long exact sequence of cohomology
\[
0\cong\mathrm{H}^1(\FF,\Oo_{\FF}(-f))\to\mathrm{H}^1(\FF,\Oo_{\FF}(h-f))\cong\CC\stackrel{\beta}{\longrightarrow} \mathrm{H}^1(h,\Oo_h(-2))\cong \CC \to \mathrm{H}^2(\FF,\Oo_{\FF}(-f))\cong 0.
\]
Then, as in Example \ref{ex_f}, the map $\beta$ sends an extension
\begin{equation}\label{extlogh2}
0\to \Oo_{\FF}(-2f)\to \Ee \to \Oo_{\FF}(-h-f) \to 0
\end{equation}
to its $h$-restriction, and the unique nontrivial extension class of %$\Ext^1_{h}(\Oo_h(2),\Oo_h)
$\Ext^1_{h}(\Oo_h,\Oo_h(-2))\cong \mathrm{H}^1(h, \Oo_h(-2))$ corresponds to the middle term $\Oo_h(-1)^{\oplus 2}$. By tensoring by $\Oo_h$ the Poincar\'e residue sequence for $\Omega_{\FF}^1(\log h)$, we obtain
\[
0\to \mathcal{T}or_{\FF}^1(\Oo_h, \Oo_h) \cong \Oo_h(1) \to \left(\Omega_{\FF}^1\right)_{|h} \to \left( \Omega_{\FF}^1(\log h) \right)_{|h}\to \Oo_h \to 0
\]
and so we have two possibilities for $\left(\Omega_{\FF}^1(\log h)\right)_{|h}$; either $\Oo_h\oplus \Oo_h(-2)$ or $\Oo_h(-1)^{\oplus 2}$. On the other hand, applying $\eta_{*}$ to the Poincar\'e residue sequence for $\Omega_{\FF}^1(\log h)$, we get
\[
0\to\eta_{*}\Omega^1_{\FF}\cong\Omega^1_{\PP^2}\to\eta_{*}\Omega_{\FF}^1(\log h)\stackrel{\nu}{\rightarrow}\Oo_p\to \mathbf{R}^1\eta_{*}\Omega_{\FF}^1\cong\Oo_p\to \mathbf{R}^1\eta_{*}\Omega_{\FF}^1(\log h)\to 0.
\]
Since $\Omega_{\PP^2}^1$ is locally free, the map $\nu$ must be trivial and so we get
\[
\eta_{*}\Omega_{\FF}^1(\log h)\cong \Omega_{\PP^2}^1 \quad \text{ and }  \quad \mathbf{R}^1\eta_{*}\Omega_{\FF}^1(\log h)\cong 0.
\]
In particular, we get $\left(\Omega_{\FF}^1(\log h)\right)_{|h}\cong \Oo_h(-1)^{\oplus 2}$. Otherwise we would get $\mathbf{R}^1\eta_{*}\Omega_{\FF}^1(\log h)\cong \Oo_p$ by the theorem on formal functions. Therefore, $\Omega_{\FF}^1(\log h)$ corresponds to a unique nontrivial extension class in $\Ext_{\FF}^1(\Oo_{\FF}(-h-f), \Oo_{\FF}(-2f))$.
\end{example}

\begin{proposition}\label{prop3.11}
$\Omega_{\FF}^1(\log h)$ is a unique point of the moduli space $\mathbf{M}_{\FF}(-h-3f,2)$ of stable bundles with respect to the polarization $H=\Oo_{\FF}(2h+3f)$.
\end{proposition}
\begin{proof}
Consider a vector bundle $\Ee$ given by the nontrivial extension $\eqref{extlogh2}$. Suppose that $\Oo_{\FF}(D)=\Oo_{\FF}(ah+bf)$ is a destabilizing line subbundle of $\Ee$. Then we have
\begin{equation}\label{distabh}
\mu_{H}(\Oo_{\FF}(D))=a+2b>-\frac{7}{2}=\mu_{H}(\Ee).    
\end{equation}
Set $g: \Oo_{\FF}(D) \subset  \Gg \rightarrow \Oo_{\FF}(-h-f)$ be the composite with the surjection in (\ref{extlogh2}). If $g$ is trivial, then we have an injection $\Oo_{\FF}(D)\hookrightarrow\Oo_{\FF}(-2f)$. In particular, the divisor $-ah-(b+2)f$ is effective and so $a\le 0$ and $b\le -2$. This is impossible due to the inequality (\ref{distabh}). If $g$ is not trivial, then we get an injection $\Oo_{\FF}(D)\hookrightarrow\Oo_{\FF}(-h-f)$. Then the divisor $-(a+1)h-(b+1)f$ is effective and so $a\le -1$ and $b\le -1$, again contradicting to (\ref{distabh}), except the cases $(a,b)\in (-1,-1)$. The case is not possible, otherwise the extension (\ref{extlogh2}) would be trivial. Thus a vector bundle from (\ref{extlogh}) is stable with respect to $H$.

Conversely, let $\Gg\in \mathbf{M}_{\FF}(-h-3f, 2)$ and then we have $\chi(\Gg(2f))=1$ by the Riemann-Roch theorem. Note that
\[
\mathrm{H}^2(\Gg(2f))\cong \mathrm{H}^0(\Gg^\vee(-2h-5f))^\vee\cong  \mathrm{H}^0(\Gg(-h-2f))^\vee
\]
by the Serre duality, and this is trivial due to the stability of $\Gg$. In particular, we have $\mathrm{h}^0(\Gg(2f))\ge 1$ and so we get an exact sequence
\[
0\to \Oo_{\FF}(-2f+D) \rightarrow \Gg \rightarrow \Ii_{Z, \FF}(-h-D) \to 0
\]
for an effective divisor $D$ and a $0$-dimensional subscheme $Z$. By the stability of $\Gg$ we should have $D=0$, and then $\Gg$ admits the extension (\ref{extlogh2}). Therefore, $\Omega_{\FF}^1(\log h)$ is a unique point contained in a one-point space $\mathbf{M}_{\FF}(-h-3f, 2)$, namely $\left(\eta^{*}\Omega^1_{\PP^2}\right)(h)$.
\end{proof}

\begin{proposition}
For an arrangement $\Dd=\{h, f_1, \ldots,f_m\}$ with the exceptional divisor $h$ and $m$ distinct fibres $f_1, \dots, f_m$, 
	%Then
%\begin{align*}
%	&\Omega^1_\FF(\log \Dd)|_h \cong \Oo_h(k)\dirsum\Oo_h(m-2-k), (\text{for some}\,\, k\geq 0)\\ 
%	&\Omega^1_\FF(\log \Dd)|_f \cong \Oo_f\dirsum\Oo_f(-1)
%\end{align*}
we have
\[
\Omega^1_{\FF}(\log \Dd)\cong \Oo_\FF((m-2)f)\dirsum\Oo_\FF(-h-f).
\]
\end{proposition}

\begin{proof}
Note that the case when $m=0$ is dealt in Example \ref{ex_h} and so we assume that $m\ge 1$. Tensoring the exact sequence
\begin{equation}\label{eq}
0\to \Omega^1_\FF\to\Omega^1_\FF(\log\Dd)\to\Oo_{h}\oplus \left(\Dirsum_{i=1}^{m}\Oo_{f_i}\right)\to 0
\end{equation}
with $\Oo_f$, we get an exact sequence
\begin{equation}\label{eqf}
0\to \mathcal{T}or_{\FF}^1(\Oo_h, \Oo_f)\oplus \left(\Dirsum_{i=1}^{m}\mathcal{T}or_{\FF}^1(\Oo_{f_i},\Oo_f)\right)\to \left(\Omega^1_{\FF}\right)_{|h}\to\left(\Omega^1_{\FF}(\log\Dd)\right)_{|f}\to\left(\Oo_h\oplus \left(\Dirsum_{i=1}^{m}\Oo_{f_i}\right)\right)\tensor\Oo_f\to 0.
\end{equation}
If $f=f_i$ for some $i$, then the sequence (\ref{eqf}) induces an exact sequence
\[
0\to\Oo_f(-2)\to \left( \Omega^1_\FF(\log\Dd)\right)_{|f}\to\Oo_q\dirsum\Oo_f\to 0,
\]
where $q:=h\cap f$. Thus we have
\[
\left( \Omega^1_\FF(\log\Dd)\right)_{|f} \cong \Oo_f\oplus \Oo_f(-1).
\]
On the other hand, if $f\neq f_i$ for any $i$, we obtain the following commutative diagram
\begin{equation}\label{ddia5}
\begin{tikzcd}[
  row sep=small,
  ar symbol/.style = {draw=none,"\textstyle#1" description,sloped},
  isomorphic/.style = {ar symbol={\cong}},
  ]
&0\ar[d] &0\ar[d]   \\
0\ar[r] &\Oo_f(1) \ar[r] \ar[d] &\Tt_f\cong \Oo_f(2)  \ar[d]\ar[r] & \Oo_q \ar[r]\ar[d, isomorphic] &0\\
0\ar[r] &\left( \Tt_{\FF}(-\log \Dd)\right)_{|f} \ar[r] \ar[d] & \left( \Tt_{\FF}\right)_{|f}  \ar[r] \ar[d] &\Oo_q \ar[r] & 0\\
&\Oo_f \ar[r, isomorphic] \ar[d]&\Nn_{f|\FF}\cong \Oo_f \ar[d]  \\
&0 &0&
\end{tikzcd}
\end{equation}
so that we still get $\left( \Omega^1_\FF(\log\Dd)\right)_{|f} \cong \Oo_f\oplus \Oo_f(-1)$. This implies in the setting of Remark \ref{pi} that the vector bundle $\Omega_\FF^1(\log\Dd)$ is $\pi$-uniform with 
\[
c_1(\Omega_\FF^1(\log\Dd))=-h+(m-3)f,\quad  c_2(\Omega_\FF^1(\log\Dd))=2-m, \quad (d,r)=(0,m-2),
\]
and so we get an exact sequence
\begin{equation}
0 \to \Oo_{\FF}((m-2)f) \to \Omega_\FF^1(\log\Dd) \to \Oo_{\FF}(-h-f) \to 0.
\end{equation}
Since we have $\Ext_{\FF}^1(\Oo_\FF(-h-f), \Oo_{\FF}((m-2)f))\cong \mathrm{H}^1(\Oo_\FF(h+(m-1)f))\cong 0$ for $m\geq 1$, the sequence in the above splits and the assertion follows.     
\end{proof}

Let $L$ be a line in $\PP^2$ with $p\not\in L$ so that $\widetilde{L}\in |\Oo_{\FF}(h+f)|$ is a rational curve with $\widetilde{L}\cap E=\emptyset$. 

\begin{proposition}\label{prop;m=1}
The logarithmic vector bundle $\Omega_{\FF}^1(\log \widetilde{L})$ is independent on the choice of $L$. 
\end{proposition}

\begin{proof}
From the commutative diagram
\begin{equation}
\begin{tikzcd}[
  row sep=small,
  ar symbol/.style = {draw=none,"\textstyle#1" description,sloped},
  isomorphic/.style = {ar symbol={\cong}},
  ]
&0\ar[d] &0\ar[d]   \\
0\ar[r] &\Oo_f(1) \ar[r] \ar[d] &\Tt_f\cong \Oo_f(2)  \ar[d]\ar[r] & \Oo_{\widetilde{L}\cap f} \ar[r]\ar[d, isomorphic] &0\\
0\ar[r] &\left( \Tt_{\FF}(-\log \widetilde{L})\right)_{|f} \ar[r] \ar[d] & \left( \Tt_{\FF}\right)_{|f}  \ar[r] \ar[d] &\Oo_{\widetilde{L}\cap f} \ar[r] & 0~,\\
&\Oo_f \ar[r, isomorphic] \ar[d]&\Nn_{f|\FF}\cong \Oo_f \ar[d] &   \\
&0 &0 &
\end{tikzcd}
\end{equation}
we have $\left(\Omega^1_{\FF}(\log \widetilde{L})\right)_{|f}\cong\Oo_f\oplus\Oo_f(-1)$ so take the invariant $d=0$. Applying $\pi_*$ to the Poincar\'e residue sequence for $\Omega^1_{\FF}(\log \widetilde{L})$, we get an exact sequence
\[
0\to\Oo_{\PP^1}(-2)\to\pi_*\Omega^1_{\FF}(\log \widetilde{L})\xrightarrow{\nu}\Oo_{\PP^1} \to \Oo_{\PP^1} \to 0.
\]
Since the last map is a surjective endomorphism between free $\Oo_{\PP^1}$-modules, $\nu$ is trivial map. Thus $\pi_*\Omega^1_{\FF}(\log \widetilde{L})\cong \Oo_{\PP^1}(-2)$, so we get $r=-2$. Now for $\Omega^1_{\FF}(\log \widetilde{L})$, 
\[
(c_1,c_2)=(-h-2f,\,\, 2)\qquad\qquad (d,r)=(0,-2).
\]
From Proposition \ref{AB} we have an extension
\begin{equation}\label{L1}
    	0 \to \Oo_{\FF}(-2f) \to \Omega_\FF^1(\log \widetilde{L}) \to \Oo_{\FF}(-h) \to 0.
\end{equation}
Note that $\Ext^1_{\FF}(\Oo_{\FF}(-h),\Oo_{\FF}(-2f))\cong\mathrm{H}^1(\Oo_{\FF}(h-2f))$ is 3-dimensional. 
Similar to Example \ref{ex_f} and Example \ref{ex_h}, we get an exact sequence
\begin{equation}\label{eq18}
\xymatrix{
    %\begin{tikzcd}[
  %row sep=small,
  %ar symbol/.style = {draw=none,"\textstyle#1" description,sloped},
  %isomorphic/.style = {ar symbol={\cong}},
  %]
0\ar[r] &\mathrm{H}^1(\FF,\Oo_{\FF}(-2f)) \ar[r] \ar[d]^-{\cong} &\mathrm{H}^1(\FF,\Oo_{\FF}(h-2f))  \ar[d]^-{\cong}\ar[r]^-{\gamma} & \mathrm{H}^1(h,\Oo_{h}(-3)) \ar[r]\ar[d]^-{\cong} &0\\
0\ar[r] &\mathrm{H}^0(\Oo_{\PP^1}) \ar[r] & \mathrm{H}^0(\Oo_{\PP^1}\dirsum\Oo_{\PP^1}(1)) \ar[r] &\mathrm{H}^0(\Oo_{\PP^1}(1)) \ar[r] & 0
}
\end{equation}
Thus the extension
\[
0 \to \Oo_h(-2) \to \Gg \to \Oo_h(1) \to 0
\]
can be lifted to the extension
\begin{equation}\label{extL1}
0 \to \Oo_{\FF}(-2f) \to \Gg' \to \Oo_{\FF}(-h) \to 0.
\end{equation}
On the other hand, tensoring $\Oo_h$ to the Poincar\'e residue sequence for $\Omega^1_{\FF}(\log \widetilde{L})$ we obtain $\left(\Omega^1_{\FF}\right)_{|h}\cong\left(\Omega^1_{\FF}(\log \widetilde{L})\right)_{|h}\cong\Oo_h(1)\dirsum\Oo_h(-2)$. Hence $\Omega^1_{\FF}(\log \widetilde{L})$ corresponds to the pre-image of a trivial extension class in $\Ext^1_{h}(\Oo_h(1),\Oo_h(-2))$. From this argument, we remark that the bundle $\Omega^1_{\FF}(\log \widetilde{L})$ is independent on the choice of curve $L$, up to isomorphism. 

Let us assume that it corresponds to a trivial class of $\Ext^1_{\FF}(\Oo_{\FF}(-h),\Oo_{\FF}(-2f))$. Then the extension is split, that is $\Omega^1_{\FF}(\log \widetilde{L})\cong\Oo_{\FF}(-2f)\oplus\Oo_{\FF}(-h)$, so that $\eta_{*}\Omega^1_{\FF}(\log \widetilde{L})\cong \Oo_{\PP^2}(-2)\oplus\Ii_{p,\PP^2}$. But after take $\eta_{*}$ to the Poincar\'e residue sequence for $\Omega^1_{\FF}(\log \widetilde{L})$, we get 
\[
0\to\Omega_{\PP^2}^1\to\eta_{*}\Omega^1_{\FF}(\log \widetilde{L})\to\Oo_{L}\to\Oo_p\stackrel{\cong}{\longrightarrow}\Oo_p\to 0.
\]
Thus $\eta_{*}\Omega^1_{\FF}(\log \widetilde{L})$ corresponds to the unique non-trivial extension class in $\Ext_{\PP^2}^1(\Oo_{L},\Omega_{\PP^2}^1)$. As known well, it is the logarithmic bundle $\Omega_{\PP^2}^1(\log L)\cong \Oo_{\PP^2}(-1)\oplus\Oo_{\PP^2}(-1)$ for a line which does not pass through $p=\eta(h)$ on $\PP^2$, a contradiction. Therefore, our bundle $\Omega^1_{\FF}(\log \widetilde{L})$ corresponds to a non-trivial class in $\Ext^1_{\FF}(\Oo_{\FF}(-h),\Oo_{\FF}(-2f))$, which maps to $0\in\Ext_h^1(\Oo_h(1),\Oo_h(-2))$. Thus it is independent on the choice of $L$.
\end{proof}

\begin{remark}
In \cite{HMPV} the Torelli problem on the lines in $\PP^2$ has a positive answer if the number $k$ of blown-up points is at least $6$. By Proposition \ref{prop;m=1}, the Torelli problem has a negative answer for $k=1$. So the natural question would be what the minimum $k$ is for a positive answer to the Torelli problem. 
\end{remark}

%%%%%%%%%%%%%%%%%%%%%%%%%%%%%%%

\bibliographystyle{abbrv}
%\bibliographystyle{ieeetr}
%\bibliography{refs}
\bibliography{citation}

%\GATHER{citation.bib}
%\providecommand{\bysame}{\leavevmode\hbox to3em{\hrulefill}\thinspace}
%\providecommand{\MR}{\relax\ifhmode\unskip\space\fi MR }
%\providecommand{\MRhref}[2]{%
%  \href{http://www.ams.org/mathscinet-getitem?mr=#1}{#2}
%}
%\providecommand{\href}[2]{#2}

%\begin{thebibliography}{9}

%\end{thebibliography}

\end{document}